\pdfoutput=1
\RequirePackage{ifpdf}
\ifpdf % We are running pdfTeX in pdf mode
\documentclass[pdftex]{sigma}
\else
\documentclass{sigma}
\fi

\usepackage{tikz}
\usetikzlibrary{arrows,calc,through,backgrounds,matrix,decorations.pathmorphing}
\tikzset{node distance=2cm, auto}
\newcommand{\Zc}{\mathcal{Z}}
\newcommand{\hmod}{{_H\mathcal{M}}}
\newcommand{\hhash}{{^{\#}_H \mathcal{M}}}
\newcommand{\Nc}{\mathcal{N}}
\newcommand{\ns}[1]{~\hspace{-3pt}_{\left<#1\right>}}
\newcommand{\id}{1}
\newcommand{\la}{\triangleright}
\newcommand{\ra}{\triangleleft}
\newcommand{\lab}{\blacktriangleright}
\newcommand{\rab}{\blacktriangleleft}
\def\ot{\otimes}

\numberwithin{equation}{section}

\newtheorem{Theorem}{Theorem}[section]
\newtheorem{Lemma}[Theorem]{Lemma}
 { \theoremstyle{definition}
\newtheorem{Definition}[Theorem]{Definition}
\newtheorem{Remark}[Theorem]{Remark} }

\begin{document}
\allowdisplaybreaks

\newcommand{\arXivNumber}{1804.02031}

\renewcommand{\PaperNumber}{098}

\FirstPageHeading

\ShortArticleName{Anti-Yetter--Drinfeld Modules for Quasi-Hopf Algebras}

\ArticleName{Anti-Yetter--Drinfeld Modules\\ for Quasi-Hopf Algebras}

\Author{Ivan KOBYZEV~$^\dag$ and Ilya SHAPIRO~$^\ddag$}

\AuthorNameForHeading{I.~Kobyzev and I.~Shapiro}

\Address{$^\dag$~Department of Pure Mathematics, University of Waterloo, Waterloo, ON N2L 3G1, Canada}
\EmailD{\href{mailto:ikobyzev@uwaterloo.ca}{ikobyzev@uwaterloo.ca}}

\Address{$^\ddag$~Department of Mathematics and Statistics, University of Windsor,\\
\hphantom{$^\ddag$}~401 Sunset Avenue, Windsor, Ontario N9B 3P4, Canada}
\EmailD{\href{mailto:ishapiro@uwindsor.ca}{ishapiro@uwindsor.ca}}

\ArticleDates{Received April 20, 2018, in final form September 10, 2018; Published online September 13, 2018}

\Abstract{We apply categorical machinery to the problem of defining anti-Yetter--Drinfeld modules for quasi-Hopf algebras. While a definition of Yetter--Drinfeld modules in this setting, extracted from their categorical interpretation as the center of the monoidal category of modules has been given, none was available for the anti-Yetter--Drinfeld modules that serve as coefficients for a Hopf cyclic type cohomology theory for quasi-Hopf algebras. This is a~followup paper to the authors' previous effort that addressed the somewhat different case of anti-Yetter--Drinfeld contramodule coefficients in this and in the Hopf algebroid setting.}


\Classification{18D10; 18E05; 19D55; 16T05}

\section{Introduction}

It is an interesting fact that the theory of coefficients, in cyclic homology theories for Hopf algebras, began with what is now known as anti-Yetter--Drinfeld \emph{modules} in \cite{hkrs1, hkrs2, js} that followed \cite{CM98, CM99}. It was not until~\cite{contra} that anti-Yetter--Drinfeld \emph{contramodules} were introduced. The latter, in retrospect, seem a lot more natural, though they involve notions that are less so.

In this followup paper to \cite{sk}, we make the modifications necessary to deal with the definitions of anti-Yetter--Drinfeld modules for quasi-Hopf algebras, which generalize Hopf algebras by relaxing the coassociativity condition to coassociativity up to a specified isomorphism. This isomorphism complicated matters sufficiently that a generalization of the formulaic approach used for Hopf algebras, via the usual method of educated guessing, is not possible. More precisely, the definition of an anti-Yetter--Drinfeld module involves a module structure together with a compatible comodule structure, and while it turns out that both the module and the compatibility remain essentially the same, the notion of a comodule needs to be rather drastically changed.

Our approach to extracting the definitions of anti-Yetter--Drinfeld modules is similar to~\cite{majid}, where the categorical notion of the center of a monoidal category of modules over a quasi-Hopf algebra has been unwound into explicit formulas. Our task is complicated by two factors: we deal with a certain bimodule category over the category of modules over a quasi-Hopf algebra and we allow not only finite dimensional representations.

The main theme of \cite{sk} was exploiting the fact that the category of modules is biclosed, i.e., it possesses internal Homs. This allows for a definition of a natural bimodule category over it, the center of which is what we are looking for. The justification for the importance of the center is the observation that its elements (or rather the ones satisfying an additional stability condition) can be used to quickly manufacture a functor called a symmetric $2$-contratrace \cite{hks} and thus define a cyclic cohomology theory for quasi-Hopf algebras. The latter is not yet available in the literature.

As mentioned above, historically anti-Yetter--Drinfeld modules appeared before their contramodule versions, but in this paper we rely on the conceptual definition of generalized anti-Yetter--Drinfeld contramodules as in~\cite{sk} to obtain the module version. In particular the question of stability is explicitly reduced to the contramodule case, though it is immediate that it is equivalent to the stability in~\cite{hks}, though not completely analogous to what is possible to do in the Hopf algebra case; see Remark~\ref{sigmadef}.

The paper is organized as follows. In Section \ref{s:General} we review and augment some generalities from \cite{sk} that we use subsequently. Section \ref{s:qHopf} is a review of the basics of quasi-Hopf algebras and recalls the definition of internal Homs, from~\cite{sk}, for them. Finally in Section \ref{s:qHopf-ayd} we unravel the conceptual definitions of Section~\ref{s:General} into formulas. More precisely, we will derive concrete definitions of anti-Yetter--Drinfeld modules in terms of structures of a module, ``comodule'', and their compatibility. An interesting observation is the appearance of two distinct ways of writing down the formulas; these are identical for Hopf algebras but very different here. This also occurs in the contramodule case~\cite{sk}.

\textbf{Notation.} We use the following particular form of Sweedler's notation. For a coalgebra $A$ (not necessarily coassociative) we denote the coproduct $\Delta(a)$ of an element $a\in A$ by $a^1\ot a^2$. We use only lowercase letters to denote elements of a coalgebra so that $S^2$ always means the square of the antipode, and never a part of a coproduct. The composition $(\operatorname{Id}\ot\Delta)\circ\Delta$ applied to $a$ is written as $a^1\ot a^{21}\ot a^{22}$, and $(\Delta\ot \operatorname{Id})\circ\Delta(a)=a^{11}\ot a^{12}\ot a^2$. On the other hand if $M$ is a~right ``comodule'' (where the quotation marks emphasize the lack of any implied coassociativity) then the structure map $\Delta\colon M\to M\ot A$ is written differently depending on the context, i.e., $\Delta(m)=m_{0}\ot m_{1}$ in the Hopf algebra case (where it is a right comodule), $\Delta(m)=m\ns{0}\ot m\ns{1}$ in the Type I case (where a modified coassociativity condition holds), and $\Delta(m)=m_{[0]} \otimes m_{[1]}$ in the Type II case (with its own version of coassociativity).

\begin{Remark}Throughout the paper we assume that the antipode $S$ is invertible. In particular, its inverse is explicitly used in~\eqref{righthom} in the definition of the right internal Hom. It has been pointed out to us by a referee that one can still define the right internal Hom in the absence of~$S^{-1}$. Indeed at this point we must treat the existence of the inverse as a possibly removable technical condition. However since non-invertible antipodes are somewhat pathological and while the presence of the inverse may not be required, its absence does complicate the exposition, we proceed as before, having conceded this point.
 \end{Remark}

\section{Generalities}\label{s:General}

In this section we will extend the general formalism of \cite{sk} from generalized anti-Yetter--Drinfeld contramodule coefficients to their module variant. Recall from \cite{hks} that the main ingredient in constructing Hopf-cyclic cohomology is a symmetric 2-contratrace. It is with a view towards this goal that we undertake the following.

 Let $\mathcal{M}$ be a biclosed monoidal category, i.e., it possesses internal Homs. More precisely, the property of being biclosed implies in particular the existence of the following adjunctions
 for $M, V, W\in \mathcal{M}$:
 \begin{gather*}
 \operatorname{Hom}_\mathcal{M} (W\otimes V, M) \simeq \operatorname{Hom}_\mathcal{M} \big(W, \operatorname{Hom}^l(V,M)\big),
 \end{gather*}
 and{\samepage
 \begin{gather*}
 \operatorname{Hom}_\mathcal{M} (V\otimes W, M) \simeq \operatorname{Hom}_\mathcal{M} \big(W, \operatorname{Hom}^r(V,M)\big),
 \end{gather*}
 where $\operatorname{Hom}^l(V,M) $ and $\operatorname{Hom}^r(V,M)$ are left and right internal homomorphisms respectively.}

As in \cite{s}, we can introduce the contragradient $\mathcal{M}$-bimodule category $\mathcal{M}^{\rm op}$. Specifically, for $M\in \mathcal{M}^{\rm op}$ and $V\in \mathcal{M}$, the actions are given by
 \begin{gather*}
 M\rab V := \operatorname{Hom}^r(V, M) \qquad \text{and} \qquad V\lab M :=\operatorname{Hom}^l(V, M) .
 \end{gather*}

A natural object to consider in this situation is the center of a bimodule category $\Nc$. Roughly speaking, it is a category with a left and a right action of our monoidal category $\mathcal{M}$, where we denote the actions by $\la$ and $\ra$ respectively. However, here it becomes too restrictive. If in the definition of a (strong) center element $N\in\Nc$, we relax the condition that, for all $N'\in\mathcal{M}$ the maps $\tau\colon N'\la N\to N\ra N'$ in $\Nc$ are isomorphisms, we get a \emph{weak} center. More formally:

 \begin{Definition} The weak center $ \text{\rm w-}\Zc_{\mathcal{M}}(\mathcal{N})$ of a~$\mathcal{M}$-bimodule category $\mathcal{N}$ consists of objects that are pairs $(N, \tau_{N,-})$, where $N \in \mathcal{N}$ and $\tau_{N,-}$ is a family of natural morphisms such that for $V \in \mathcal{M}$ we have: $\tau_{N,V}\colon V \la N \to N \ra V $, satisfying the hexagon axiom as in the usual definition of center (see~\cite{ENO}).
 \end{Definition}

 We need one more definition.

\begin{Definition}Let $\mathcal{M} $ be a biclosed monoidal category and $\mathcal{M}^{\rm op}$ the contragradient category as above. For $M\in \text{\rm w-}\Zc_{\mathcal{M}}(\mathcal{M}^{\rm op})$ let
\begin{gather*}
\sigma_M\colon \ M\to M
\end{gather*} denote the image of $\operatorname{Id}_M$ under \begin{gather*}\operatorname{Hom}_\mathcal{M}(M,M)\simeq\operatorname{Hom}_\mathcal{M}(\id,M\lab M)\to\operatorname{Hom}_\mathcal{M}(\id,M\rab M)\simeq \operatorname{Hom}_\mathcal{M}(M,M),\end{gather*}
where the map in the middle is postcomposition with $\tau$ and the isomorphisms come from the definitions of the actions in $\mathcal{M}^{\rm op}$.
\end{Definition}

We have from \cite{sk}:
\begin{Lemma}\label{weakstrong} If $M\in \text{\rm w-}\Zc_{\mathcal{M}}(\mathcal{M}^{\rm op})$ and $\sigma_M=\operatorname{Id}$ then $M\in \Zc_{\mathcal{M}}(\mathcal{M}^{\rm op})$.
 \end{Lemma}

 In \cite{sk}, we called such $M$ stable and denoted the full subcategory containing them by $\Zc_{\mathcal{M}}'(\mathcal{M}^{\rm op})$. These are exactly the generalized stable anti-Yetter--Drinfeld contramodules, and the functor $\operatorname{Hom}_{\mathcal{M}}(-,M)$ is a symmetric $2$-contratrace. Recall from \cite{hks}, that once we have a~contratrace, the cyclic cohomology theory immediately follows. More precisely, for an algebra $A\in\mathcal{M}$ the collection $\operatorname{Hom}_{\mathcal{M}}(A^{\ot\bullet+1},M)$ is naturally a~cocyclic vector space from which the cohomology is obtained.

 \subsection{The module variant modifications}
 Assume as above that $\mathcal{M}$ is biclosed and suppose further that there exists a tensor auto-equivalence $(-)^{\#}$ of $\mathcal{M}$, together with natural identifications
 \begin{gather*}
 W\lab 1\simeq 1\rab W^{\#}.
 \end{gather*}
 Observe that should such a functor exist, we would immediately have natural identifications
 \begin{gather*}
 \iota_{V,W}\colon \ \operatorname{Hom}_\mathcal{M}(V\ot W, 1)\simeq \operatorname{Hom}_\mathcal{M}\big(W^{\#}\ot V, 1\big)
 \end{gather*}
 for all $V,W\in\mathcal{M}$.

 Consider ${^{\#}\mathcal{M}}$, the $\mathcal{M}$-bimodule category with the right and left $\mathcal{M}$-module structures given by
 \begin{gather*}
 M \ra V= M\otimes V \qquad \text{and} \qquad V \la M= V^\# \otimes M .
 \end{gather*}

 \begin{Lemma}\label{dfun}We have the functor between weak centers
\begin{gather*}
D\colon \ \big(\text{\rm w-}\Zc_{\mathcal{M}}({^{\#}\mathcal{M}})\big)^{\rm op}\to \text{\rm w-}\Zc_{\mathcal{M}}\big(\mathcal{M}^{\rm op}\big)
\end{gather*}
that sends $M$ to $1\rab M$.
 \end{Lemma}
 \begin{proof}
The weak center structure on $1\rab M$ is obtained as follows: $V\lab(1\rab M)\simeq (V\lab 1)\rab M\simeq (1\rab V^{\#})\rab M\simeq 1\rab(V^{\#}\ot M)\to 1\rab(M\ot V)\simeq(1\rab M)\rab V$.
 \end{proof}

 Recall the definition of stability for generalized anti-Yetter--Drinfeld modules from \cite{hks}. For $M\in\text{\rm w-}\Zc_{\mathcal{M}}({^{\#}\mathcal{M}})$, if \begin{gather*}
 \operatorname{Hom}_\mathcal{M}(M\otimes V,1) \xrightarrow{-\circ\tau} \operatorname{Hom}_\mathcal{M}(V^\# \otimes M,1) \xrightarrow{\iota^{-1}_{M,V}} \operatorname{Hom}_\mathcal{M}(M\otimes V,1)
 \end{gather*}
 is the identity, then $M$ is called stable. It is immediate that the following definition is equivalent to this one.

 \begin{Definition}\label{stab}
Let $M\in\text{\rm w-}\Zc_{\mathcal{M}}({^{\#}\mathcal{M}})$, we say that $M$ is stable if $DM\in\text{\rm w-}\Zc_{\mathcal{M}}(\mathcal{M}^{\rm op})$ is stable, i.e., if $\sigma_{1\rab M}=\operatorname{Id}$.
 \end{Definition}

 We have the following analogue of Lemma \ref{weakstrong}:

\begin{Lemma}Suppose that $M\in\text{\rm w-}\Zc_{\mathcal{M}}({^{\#}\mathcal{M}})$ is stable. Assume that $D$ reflects isomorphisms $($when considered as a functor from $({^{\#}\mathcal{M}})^{\rm op}$ to $\mathcal{M}^{\rm op})$, then $M\in\Zc_{\mathcal{M}}({^{\#}\mathcal{M}})$. We will denote the full subcategory of such $M$ by $\Zc'_{\mathcal{M}}({^{\#}\mathcal{M}})$.
 \end{Lemma}
\begin{proof}By definitions if $M$ is stable, then $\sigma_{1\rab M}=\operatorname{Id}$, and so by Lemma \ref{weakstrong} and the proof of Lemma \ref{dfun} the centrality map $\tau\colon V^{\#}\ot M\to M\ot V$ is such that its right dual is an isomorphism, i.e., $D(\tau)$ is an isomorphism. If $D$ reflects isomorphisms then $\tau$ is an isomorphism as well.
\end{proof}

The $M$ of the lemma above are exactly the generalized stable anti-Yetter--Drinfeld modules. We note that the functor $\operatorname{Hom}_\mathcal{M}(M\ot-,1)$ is a symmetric $2$-contratrace in this case. This follows immediately from its isomorphism to $\operatorname{Hom}_\mathcal{M}(-,DM)$ as $DM\in\Zc_{\mathcal{M}}'(\mathcal{M}^{\rm op})$.

 \begin{Remark}\label{sigmadef}
The current approach to general stability of aYD modules, via stability of aYD contramodules (implicitly as in \cite{hks} or explicitly as in Definition \ref{stab}), may seem unsatisfactory but it is the only way in general. In particular situations one may do better. More precisely, it may happen that for $M\in\text{\rm w-}\Zc_{\mathcal{M}}({^{\#}\mathcal{M}})$ the map $\sigma_{DM}$ may have a predual, i.e., a $\sigma_M$ such that $\operatorname{Id}\rab\sigma_M=\sigma_{1\rab M}$. This is the case for Hopf algebras where $\sigma_M(m)=m_{1}m_{0}$, but this question remains open in the quasi-Hopf algebra case.
 \end{Remark}

\section{Recalling quasi-Hopf algebras}\label{s:qHopf}
 Let us remind the reader of all the necessary definitions following~\cite{Drinfeld}. In this section, $k$ is a field.
\begin{Definition}A quasi-bialgebra is a collection $(A, \Delta, \varepsilon, \Phi)$, where $A$ is an associative $k$-algebra with unity, $\Delta\colon A \rightarrow A\otimes A$ and $\varepsilon\colon A \rightarrow k $ are homomorphisms of algebras, $\Phi \in A \otimes A\otimes A$ is an invertible element, such that the following equalities hold:
\begin{gather}
(\operatorname{Id} \otimes \Delta) (\Delta (a) ) = \Phi \cdot ( (\Delta \otimes \operatorname{Id}) (\Delta(a) ) ) \cdot \Phi^{-1}, \qquad \forall\, a \in A ,\label{coassoc}\\
\lbrack (\operatorname{Id} \otimes \operatorname{Id} \otimes \Delta)(\Phi) \rbrack \cdot \lbrack (\Delta \otimes \operatorname{Id} \otimes \operatorname{Id})(\Phi) \rbrack = (1 \otimes \Phi) \cdot \lbrack (\operatorname{Id} \otimes \Delta \otimes \operatorname{Id})(\Phi) \rbrack \cdot (\Phi \otimes 1),\label{phi} \\
 (\varepsilon \otimes \operatorname{Id} ) (\Delta (a)) = a, \qquad (\operatorname{Id} \otimes \varepsilon ) (\Delta (a)) = a, \qquad \forall\, a \in \mathcal{A}, \label{one} \\
\operatorname{Id} \otimes \varepsilon \otimes \operatorname{Id} (\Phi) = 1 \otimes 1. \label{unass}
\end{gather}
\end{Definition}

 \begin{Remark}\label{Sweedler}In this paper we will use the Sweedler notation. Let's denote
 \begin{gather*}
 \Phi = X \otimes Y \otimes Z, \qquad \Phi^{-1} = P \otimes Q \otimes R ,
 \end{gather*}
here we mean the summation. In particular, the equality \eqref{coassoc} can be written as
\begin{gather*}
a^1\otimes a^{21}\otimes a^{22} = Xa^{11}P\otimes Ya^{12}Q \otimes Z a^2 R.
\end{gather*}
\end{Remark}

We are interested in the category of left $A$-modules $_A\mathcal{M}$. It was proved in~\cite{Drinfeld} that this category is monoidal if a tensor product of two left $A$-modules $M$ and $N$ is defined by the same formula as in the case of a bialgebra
\begin{gather*}
 M \otimes N = M \otimes_k N, \qquad a\cdot (m\otimes n) = a^1m\otimes a^2n.
\end{gather*}

The associativity morphism is no longer trivial as it was in the case of a bialgebra. If one sets the associativity morphism $(M \otimes N)\otimes L \rightarrow M \otimes (N\otimes L)$ to be the image of $\Phi$ in $\operatorname{End}_k(M \otimes N \otimes L)$, then it becomes an isomorphism of left $A$-modules by~\eqref{coassoc}. Consider $k$ as an $A$-module by $a \cdot 1 = \varepsilon(a) 1 $ as in the bialgebra case. Then one defines a morphism $\lambda_M\colon k \otimes M \rightarrow M $ as the usual morphism of $k$-modules. Then $\lambda_M$ becomes an $A$-module morphism by~\eqref{one}. Similarly one can define a morphism $\rho_M\colon M\otimes k \rightarrow M $. So $k$ is a unit of the monoidal category $_A\mathcal{M}$.

\begin{Remark}The equality \eqref{phi} is equivalent to the pentagon axiom. Associativity and unit in the category respect each other by \eqref{unass}.
 \end{Remark}

Recall the definition of a quasi-Hopf algebra from \cite{Drinfeld}.

\begin{Definition} Let $(H, \Delta, \varepsilon, \Phi)$ be a quasi-bialgebra. Then it is called a quasi-Hopf algebra if there exist $\alpha, \beta \in H $ and anti-automorphism $S\colon H \rightarrow H$, such that
 \begin{gather*}
 S\big(h^1\big)\alpha h^2 = \varepsilon(h) \alpha, \qquad h^1 \beta S\big(h^2\big) = \epsilon(h) \beta .
 \end{gather*}

If one keeps notation as in Remark \ref{Sweedler}, then there are equalities
 \begin{gather*}
 X \beta S(Y) \alpha Z = 1, \qquad S(P) \alpha Q \beta R = 1.
 \end{gather*}
\end{Definition}

\begin{Remark}	We want to emphasize that the antipode $S$ in the definition above is assumed to be invertible.
\end{Remark}

It was shown in \cite{Drinfeld} that the category of $H$-modules for a quasi-Hopf algebra $H$ that are finite dimensional as $k$-vector spaces is rigid. It was shown in \cite{sk} (see also \cite{BPO_closed, majid_closed, SS_closed}) that if we consider all modules, i.e., $\hmod$ then it is biclosed. More precisely, for any $M, N\in \hmod$, we can define the left internal Hom by
\begin{gather*} %\label{lefthom}
 \operatorname{Hom}^l(M, N)=\operatorname{Hom}_k(M, N),\qquad h\cdot \varphi = h^1\varphi\big(S\big(h^2\big){-}\big).
\end{gather*} Furthermore, we can define the right internal Hom by
\begin{gather} \label{righthom}
 \operatorname{Hom}^r(M, N)=\operatorname{Hom}_k(M, N),\qquad h\cdot \varphi = h^2\varphi\big(S^{-1}\big(h^1\big){-}\big).
\end{gather}

\section{Anti-Yetter--Drinfeld modules for a quasi-Hopf algebra}\label{s:qHopf-ayd}

Yetter--Drinfeld modules for a quasi-Hopf algebra were described in \cite{majid}. The consideration of YD modules as centers of a monoidal category~$\hmod$ was crucial to write down the explicit formulas that generalize those that define the YD modules of the Hopf algebra case. In this section we are going to define the anti-Yetter--Drinfeld modules using the categorical approach from~\cite{hks}. Unlike the Hopf-case, there are two different ways to define aYD modules. The resulting categories are isomorphic to each other, and in turn isomorphic to the center of a~certain bimodule category.

As in \cite{hks} consider a monoidal functor $\#\colon \hmod \rightarrow \hmod $, taking a left $H$-module $M$ to $M^{\#}$, where $M^\#$ is the same as $M$ as a $k$-vector space but the module structure is modified by~$S^2$
 \begin{gather*}
 h \cdot m = S^2(h)m, \qquad \text{for} \quad m \in M^\#.
 \end{gather*}

 Define the $\hmod$-bimodule category $\hhash$ such that it has the same objects as $\hmod$. For $M\in\hhash$ and $V\in\hmod$, the right and left $\hmod$-module structures are given by
 \begin{gather*}
 M \ra V= M\otimes V \qquad \text{and} \qquad V \la M= V^\# \otimes M .
 \end{gather*}

 \begin{Remark}Observe that as required according to Section~\ref{s:General} we have a very trivial identification $V\lab k\simeq k\rab V^{\#}$, indeed the left hand side is $\operatorname{Hom}_k(V,k)$ with $h\cdot\varphi=\varphi(S(h)-)$ whereas the right hand side is $\operatorname{Hom}_k(V^{\#},k)$ with $h\cdot\varphi=\varphi(S^{-1}(h)\cdot-)=\varphi(S(h)-)$. Furthermore, the functor $D$, being essentially a vector space duality functor, reflects isomorphisms, and so, as long as we insist on stability in the sense of Definition~\ref{stab}, we do not need to worry about the difference between the weak and the strong center.
\end{Remark}

If $H$ is a Hopf algebra, it was proved in \cite{hks} that the center $\Zc_\hmod(\hhash) $ is the same as anti-Yetter--Drinfeld modules. We are going to use this fact as a guide and give a description of aYD modules in the quasi-Hopf case.

\subsection{Anti-Yetter--Drinfeld modules I} Below we use a similar approach to \cite{majid} in defining the ``coaction'' from the centralizing isomorphism.

\begin{Lemma}\label{aYD-module lemma}Let $M \in \hhash$. Natural transformations $\tau_{\bullet} \in \operatorname{Nat} (\operatorname{Id} \la M, M \ra \operatorname{Id})$ are in $1$-$1$ correspondence with $k$-linear maps $\rho\colon M \rightarrow M \otimes H $, denoted by $m \mapsto m_{\langle0\rangle} \otimes m_{\langle1\rangle}$, such that
\begin{gather}\label{aYD-module}
h^1 m_{\langle 0\rangle} \otimes h^2 m_{\langle 1\rangle} = \big(h^2 m\big)_{\langle0\rangle} \otimes \big(h^2 m\big)_{\langle 1\rangle} S^2\big(h^1\big).
\end{gather}
\end{Lemma}

\begin{proof}Consider the morphism $\tau_H \colon H \la M \rightarrow M \ra H $. Define $\rho (m)$ as $\tau_H (1\la m) \in M\otimes H$. Because $\tau_H $ is a morphism in the category (of left $H$-modules), we have: $h \cdot \tau_H(1\la m) = \tau_H (h\cdot (1\la m)) $ for any $h \in H$. The left hand side gives us $h^1 m_{\langle 0\rangle} \otimes h^2 m_{\langle 1\rangle} $. Using the definition of $\hhash$ we can see that $h\cdot (1\la m) = S^2(h^1) \la h^2 m $. The right action of $H$ on itself is a~morphism in $\hmod$, so because $\tau_\bullet$ is a natural transformation, one can rewrite the right hand side as $(h^2 m)_{\langle0\rangle} \otimes (h^2 m)_{\langle 1\rangle} S^2(h^1)$.

Conversely, given the map~$\rho$, for any $V \in \hmod $ one can define the map $\tau_V\colon V \la M \rightarrow M \ra V $ by the rule $v \la m \mapsto m_{\langle0\rangle} \ra m_{\langle1\rangle} v$. It is a morphism of $H$-modules by~\eqref{aYD-module}. Clearly $\tau_\bullet$ is a natural transformation.

These correspondences are mutually inverse. For a given $v \in V$ consider a morphism $H \rightarrow V$ by the rule $h \mapsto h \cdot v$. So, if we want $\tau$ to be a natural transformation, it is uniquely defined from~$\rho$.
\end{proof}

In the Hopf case the $k$-linear map $\rho\colon M \rightarrow M\otimes H$ was a part of a right comodule structure on the left $H$-module $M$. But in the quasi-Hopf case the comodule condition must be replaced (see~\eqref{quasi-comodule} below), since we are now dealing with an~$H$ that is not coassociative (see~\cite{pv} for example).

Let $(M, \tau_\bullet) \in \Zc_\hmod(\hhash)$. By the hexagon axiom of the center, the following diagram is commutative
\[
\begin{tikzpicture}
\matrix (m) [matrix of math nodes, row sep=4em, column sep=8em]
{V \la (W \la M)& V \la (M \ra W) & (V \la M) \ra W \\
	(V\otimes W) \la M& M \ra (V\otimes W) & (M \ra V) \ra W. \\};
\draw[->] (m-1-1) to node {$\operatorname{Id} \la \tau_W $} (m-1-2);
\draw[->] (m-1-2) to node {$\Phi^{-1} $} (m-1-3);
\draw[->] (m-1-3) to node {$\tau_V \ra \operatorname{Id}$} (m-2-3);
\draw[->] (m-2-2) to node {$\Phi^{-1}$} (m-2-3);
\draw[->] (m-1-1) to node {$\Phi^{-1}$} (m-2-1);
\draw[->](m-2-1) to node {$\tau_{V\otimes W} $}(m-2-2);
\end{tikzpicture}
\]

Consider this diagram in the case $V = W = H$. Start with the element $1 \la 1 \la m$ in the upper left corner. The definition of the $k$-linear map $\rho\colon M \rightarrow M\otimes H$ from Lemma~\ref{aYD-module lemma} and the definition of the category $\hhash$ imply the following equality
 \begin{gather}
 (Qm_{\langle 0\rangle})_{\langle 0\rangle} \otimes (Qm_{\langle 0\rangle})_{\langle 1\rangle} S^2(P) \otimes R m_{\langle 1\rangle}\nonumber\\ \qquad{} = P(Rm)_{\langle 0\rangle} \otimes Q((Rm)_{\langle 1\rangle})^1S^2(P) \otimes R((Rm)_{\langle 1\rangle})^2 S^2(Q).\label{quasi-comodule}
 \end{gather}

\begin{Remark}Notice that if $H$ is Hopf algebra, $\Phi$ is trivial and the condition~\eqref{quasi-comodule} comes down to the definition of a right $H$-comodule.
\end{Remark}

Consider the following diagram
\[
\begin{tikzpicture}
\matrix (m) [matrix of math nodes, row sep=4em, column sep=8em]
{H\la M & k\la M & M \\
	 M \ra H & M \ra k & M. \\};
\draw[->] (m-1-1) to node {$ \varepsilon \la \operatorname{Id} $} (m-1-2);
\draw[->] (m-1-2) to node {$\cong $} (m-1-3);
\draw[->] (m-1-3) to node {$\operatorname{Id} $} (m-2-3);
\draw[->] (m-1-2) to node {$ \tau_k $} (m-2-2);
\draw[->] (m-2-2) to node {$ \cong $} (m-2-3);
\draw[->] (m-1-1) to node {$ \tau_H $} (m-2-1);
\draw[->](m-2-1) to node {$ \operatorname{Id} \ra \varepsilon $}(m-2-2);
\end{tikzpicture}
\]

The left square commutes by the naturality of $\tau$ and the right square commutes by the definition of the center. If we start with $1 \la m$ in the upper left corner we will get the equality
\begin{gather} \label{comodule_unit}
 m = \varepsilon (m_{\langle 1\rangle}) m_{\langle0\rangle} .
\end{gather}
 This condition is exactly the same as in the Hopf case.

\begin{Definition} Let $H$ be a quasi-Hopf algebra. A pair $(M, \rho)$, where $M$ is a left $H$-module and $\rho\colon M \rightarrow M\otimes H$ is a $k$-linear map, written as $\rho(m) = m_{\langle 0\rangle} \otimes m_{\langle 1\rangle} $, is called a left-right anti-Yetter--Drinfeld module of type I, if it satisfies the equalities \eqref{aYD-module}, \eqref{quasi-comodule} and \eqref{comodule_unit}.

A morphism of two ${\rm aYD}$ modules $(M, \rho) \rightarrow (M', \rho')$ is an $H$-morphism $f\colon M\rightarrow M'$, such that $\rho'\circ f = (f\otimes \operatorname{Id})\circ \rho $.
 \end{Definition}

\begin{Theorem}The category of aYD-modules of type I for a quasi-Hopf algebra $H$ is equivalent to the weak center $\text{\rm w-}\Zc_\hmod(\hhash)$.
\end{Theorem}
\begin{proof} We have seen that an object in the center $(M, \tau)$ gives us an aYD module $(M, \rho)$. Consider a morphism of central objects $f\colon (M, \tau) \rightarrow (M', \tau')$. In particular the following diagram must commute:
\[
\begin{tikzpicture}
\matrix (m) [matrix of math nodes, row sep=4em, column sep=8em]
{H\la M & M \ra H \\
	H\la M' & M' \ra H. \\};
\draw[->] (m-1-1) to node {$ \tau_H $} (m-1-2);
\draw[->] (m-1-2) to node {$f\ra \operatorname{Id} $} (m-2-2);
\draw[->] (m-1-1) to node {$\operatorname{Id}\la f $} (m-2-1);
\draw[->](m-2-1) to node {$ \tau'_H $}(m-2-2);
\end{tikzpicture}
\]
By the construction of $\rho$ one gets the condition $\rho'\circ f = (f\otimes \operatorname{Id})\circ \rho $.

Conversely, take an aYD module $(M, \rho)$. By Lemma \ref{aYD-module lemma}, there is a natural transformation $\tau\colon \operatorname{Id} \la M \rightarrow M \ra \operatorname{Id}$. Formula~\eqref{quasi-comodule} guarantees that~$\tau$ satisfies the hexagon axiom. Equality~\eqref{comodule_unit} gives that $\tau_k = \operatorname{Id}$.
 \end{proof}

\subsection{Anti-Yetter--Drinfeld modules II}

There is the second way to introduce aYD modules. Let us again consider the central element $(M, \tau) \in \Zc_\hmod(\hhash) $. So for any $V \in \hmod$ we have a natural isomorphism: $\tau_V\colon V^\# \otimes M \rightarrow M \otimes V $. Using internal Homs this gives a morphism
 \begin{gather*}
 \widehat{\tau}_V\colon \ M \rightarrow \operatorname{Hom}^r \big(V^\#, M \otimes V\big).
 \end{gather*}

Now we want to introduce a new $H$-module $M \otimes^r H$ which is the same as $M\otimes_k H$ as a vector space, but the $H$-action is different
 \begin{gather*}
 x \cdot (m \otimes h) = x^{21} m \otimes x^{22}h S\big(x^1\big).
 \end{gather*}

 For any $V \in \hmod$ define a map $r_V\colon M \otimes^r H \rightarrow \operatorname{Hom}^r(V^\#, M \otimes V) $ by the rule: $m\otimes h \mapsto (v \mapsto m \otimes hv) $. This is a morphism in the category by construction of $M \otimes^r H $ and $V^\#$.

 We can formulate a lemma similar to Lemma~\ref{aYD-module lemma}.

 \begin{Lemma}	Let $M \in \hhash$. Natural transformations $\tau \in \operatorname{Nat} (\operatorname{Id} \la M, M \ra \operatorname{Id})$ are in $1$-$1$ correspondence with $k$-linear maps $\lambda\colon M \rightarrow M \otimes H $, written as $m \mapsto m_{[0]} \otimes m_{[1]}$, such that
\begin{gather}\label{aYD-moduleII}
(h m)_{[0]} \otimes (h m)_{[1]} = h^{21} m_{[0]} \otimes h^{22} m_{[1]} S\big(h^1\big).
\end{gather}
\end{Lemma}

 \begin{proof}
First assume that $\tau $ is given. Then for $m \in M$ we define $\lambda(m) := (\widehat{\tau}_H (m)) (1) $. From the fact that $(\widehat{\tau}_H (m)) $ is a right internal homomorphism we get~\eqref{aYD-moduleII}.

Conversely, given a $k$-linear map $\lambda\colon M \rightarrow M \otimes H $ we can consider it as an $H$-homomorphism $M \rightarrow M \otimes^r H$ by~\eqref{aYD-moduleII}. Now for any $V \in \hmod$ we define $\widehat{\tau}_V$ by
\[
	\begin{tikzpicture}
	\matrix (m) [matrix of math nodes, row sep=3em, column sep=5em]
	{M & \operatorname{Hom}^r(V^\#, M\otimes V) \\
		M \otimes^r H. \\};
	\draw[->] (m-1-1) to node {$\widehat{\tau}_V $} (m-1-2);
	\draw[->] (m-1-1) to node {$\lambda$} (m-2-1);
	\draw[->](m-2-1) to node {$r_V $}(m-1-2);
	\end{tikzpicture}
\]
Everything is constructed naturally and the one-to-one correspondence is clear.	
 \end{proof}

 \begin{Remark}
It will be useful to explicitly write the reconstruction formula. Given $\lambda\colon M \rightarrow M\otimes H$, satisfying \eqref{aYD-moduleII}, $\tau_V \colon V^\# \otimes M \rightarrow V \otimes M$ is built by the formula
\begin{gather}\label{lambda}
v \otimes m \mapsto R^1 m_{[0]} \otimes R^2 m_{[1]} S(Q) S(\alpha) S^2(P) v.
\end{gather}
 \end{Remark}

As above to write the replacement of the comodule condition consider the hexagon axiom. Then, using formula~\eqref{lambda}, we get the following equality
 \begin{gather}
 PR^1(Rm)_{[0]} \otimes Q\big(R^2(Rm)_{[1]} \kappa \big)^1S^2(P) \otimes R \big(R^2(Rm)_{[1]} \kappa \big)^2 S^2(Q)\nonumber\\
\qquad{} =R^1 \big(QR^1m_{[0]}\big)_{[0]} \otimes R^2 \big(QR^1m_{[0]}\big)_{[1]} \kappa S^2(P) \otimes RR^2m_{[1]} \kappa, \label{quasi-comoduleII}
 \end{gather}
 where we set $\kappa = S(Q) S(\alpha) S^2(P)$.

 The unital condition that $\tau_k\colon k \la M \rightarrow M \ra k$ is identity gives
 \begin{gather*}
 m = R^1 m_{[0]} \cdot \varepsilon\big(R^2 m_{[1]} \kappa\big) = \varepsilon(m_{[1]}) R m_{[0]} \cdot(\kappa).
 \end{gather*}
 Here we used that $\varepsilon$ is algebra map and the formula~\eqref{one}. For a quasi-Hopf algebra the equality $\varepsilon \circ S = \varepsilon $ holds (for the proof see~\cite{Drinfeld}). So we can simplify $\varepsilon(\kappa) = \varepsilon(\alpha) \varepsilon(Q) \varepsilon (P) $. Using~\eqref{unass}, we get the final equality
\begin{gather} \label{comodule_unitII}
 m = \varepsilon(m_{[1]}) m_{[0]} \varepsilon(\alpha).
\end{gather}

 \begin{Definition}Let $H$ be a quasi-Hopf algebra. A pair $(M, \lambda)$, where $M$ is a left $H$-module and $\lambda\colon M \rightarrow M\otimes H$ is a $k$-linear map, written as $\lambda(m) = m_{[0]} \otimes m_{[1]} $, is called a left-right anti-Yetter--Drinfeld module of type II, if it satisfies the equalities~\eqref{aYD-moduleII}, \eqref{quasi-comoduleII} and \eqref{comodule_unitII}.
 \end{Definition}

 And as before we have the following theorem with a proof that is very similar to the type I case and so is omitted.

\begin{Theorem}The category of aYD-modules of type~II for a quasi-Hopf algebra $H$ is equivalent to $\text{\rm w-}\Zc_\hmod(\hhash)$.
 \end{Theorem}

\begin{Remark}Type I and type II aYD modules are different, though of course are equivalent as categories. The difference between them is like the difference between two maps: $H \otimes \operatorname{Hom}^r(H,V) \rightarrow V $, where the first map is $ h \otimes f \mapsto f(h)$ (naive evaluation) and the second one is $\text{\rm ev}^r(h\otimes f) $ (actual evaluation). In the Hopf case, these two evaluations are the same, but this is no longer true for a quasi-Hopf algebra.

The reader is also invited to consult the proof \cite[Lemma~2.2]{hks} that, in the Hopf case, the formula~\eqref{aYD-module} is equivalent to~\eqref{aYD-moduleII}.
 \end{Remark}

\subsection*{Acknowledgements}

The authors wish to thank Masoud Khalkhali for stimulating questions and discussions. The research of the second author was supported in part by the NSERC Discovery Grant number 406709. In addition we thank the referees for their helpful corrections.

\pdfbookmark[1]{References}{ref}
\LastPageEnding

\end{document}